\documentclass[a4paper, 11pt]{article}
\usepackage{amsmath,amssymb,color}
\usepackage[dvips]{graphicx}
\usepackage{verbatim}
\usepackage{amsthm}
\usepackage{amsmath,amssymb,color}

\RequirePackage[OT1]{fontenc}
\RequirePackage{amsthm,amsmath, amssymb, enumerate}
\RequirePackage[round, sort&compress, authoryear]{natbib}
\RequirePackage[colorlinks,citecolor=blue,urlcolor=blue]{hyperref}

\usepackage{eufrak}
\usepackage{graphicx, amsthm, url,pstricks,fancyhdr}
\RequirePackage{bbm,amsthm}
\RequirePackage{amssymb}

\RequirePackage{enumitem}
\RequirePackage{bigints}
\RequirePackage{mathtools}
\RequirePackage[stable]{footmisc}
\usepackage{bm}

\theoremstyle{plain}
\newtheorem{theorem}{Theorem}[section]                                          
\newtheorem{proposition}[theorem]{Proposition}                          
\newtheorem{lemma}[theorem]{Lemma}
\newtheorem{corollary}[theorem]{Corollary}

\theoremstyle{definition}

\theoremstyle{remark}
\newtheorem{remark}[theorem]{Remark}

\makeatletter \@addtoreset{equation}{section} \makeatother

\newcommand{\Prob}{\mathbb{P}\,}

\newcommand{\R}{\mathbb{R}}

\newcommand{\Xb}{\boldsymbol{X}}
\newcommand{\xb}{\boldsymbol{x}}

\def\heiko{\mathrel{\raise.3ex\hbox{\scalebox{.7}{%
    \rotatebox[origin=c]{-7}{/}%
    \kern-.35em\rotatebox[origin=c]{-7}{/}}}}}%

\title{Bounds for the asymptotic normality of the maximum likelihood estimator using the Delta method} 
\author{Andreas Anastasiou$^{1}$ and Christophe Ley$^{2}$
\\
\small{$^{1}$ Oxford University, Oxford, UK}
\\
\small{E-mail: andreas.anastasiou@jesus.ox.ac.uk} 
\\ 
\small{${}^{2}$ Ghent University, Gent, Belgium}
\\ 
\small{E-mail: christophe.ley@ugent.be}
}
\date{}
\begin{document}
\pagenumbering{roman}
\pagenumbering{arabic}
\maketitle

\begin{abstract}
\indent The asymptotic normality of the Maximum Likelihood Estimator (MLE) is a cornerstone of statistical theory. In the present paper, we provide  sharp explicit upper bounds on Zolotarev-type distances between the exact, unknown distribution of the MLE and its limiting normal distribution. Our approach to this fundamental issue is based on a sound combination of the Delta method, Stein's method, Taylor expansions and conditional expectations, for the classical situations where the MLE can be expressed as a function of a sum of independent and identically distributed terms. This encompasses in particular the broad exponential family of distributions. \end{abstract}


\section{Introduction}
\label{sec:intro}

The asymptotic normality of maximum likelihood estimators (MLEs)  is one of the best-known and most fundamental results in mathematical statistics. Under certain regularity conditions (given later in this section), we have the following classical theorem, first discussed in~\cite{Fisher}.
\begin{theorem}[Asymptotic Normality of the MLE]\label{theoMLEnor}
Let $X_1,\ldots,X_n$ be \mbox{i.i.d.} random variables with probability density (or mass) function $f(x_i|\theta)$, where $\theta$ is a scalar parameter. Its true value is denoted as $\theta_0$. Assume that the MLE exists and it is unique and conditions (R1)-(R4), see below, are
satisfied. Then
$$
\sqrt{ni(\theta_0)}\left(\hat\theta_n(\Xb)-\theta_0\right)\xrightarrow[{n \to \infty}]{{\rm d}}\mathcal{N}(0,1)
$$
where $i(\theta_0)$ is the expected Fisher information quantity and $\stackrel{\rm d}{\rightarrow}$ means convergence in distribution.
\end{theorem}
The aim of the present paper is to complement this qualitative result with a quantitative statement, in other words, to find the best possible approximation for the distance, at finite sample size $n$, between the distribution of $\sqrt{ni(\theta_0)}\left(\hat\theta_n(\Xb)-\theta_0\right)$ on the one hand and $\mathcal{N}(0,1)$ on the other hand. In mathematical terms, for $Z \sim \mathcal{N}(0,1)$, we are interested in the quantity
\begin{eqnarray}
&&d_H\left(\sqrt{ni(\theta_0)}\left(\hat\theta_n(\Xb)-\theta_0\right),Z\right)\nonumber\\
&&=\sup_{h\in H}\left|{\rm E}\left[h\left(\sqrt{ni(\theta_0)}\left(\hat\theta_n(\Xb)-\theta_0\right)\right)\right]-{\rm E}\left[h\left(Z\right)\right]\right|\label{aim}
\end{eqnarray}
with  
\begin{equation}
\label{classfunctions}
H = \left\lbrace h:\mathbb{R}\rightarrow\mathbb{R},\,{\rm \;absolutely\;continuous\;and\;bounded}\right\rbrace.
\end{equation}
Distances of this type are called \emph{Zolotarev-type distances}. Our main focus will lie on the classical situations where the MLE can be expressed as a function of a sum of independent and identically distributed terms.

Consider an \mbox{i.i.d.} sample of observations $\Xb=(X_1,\ldots,X_n)$. Writing $\hat{\theta}_n(\Xb)$ the MLE of the scalar parameter of interest $\theta\in\Theta\subseteq\R$, we are interested in settings where there exists a one-to-one twice differentiable mapping $q:\Theta\rightarrow\R$ such that 
\begin{equation}\label{MLEfunc}
q\left(\hat{\theta}_n(\Xb)\right)=\frac{1}{n}\sum_{i=1}^ng(X_i)
\end{equation}
for some $g:\R\rightarrow\R$. Situations of this kind are all but rare; with $f(x|\theta)$ the probability density (or mass) function, classical examples include 
\begin{itemize}
\item the normal distribution with density $f(x|\mu, \sigma^2)=\frac{1}{\sigma\sqrt{2\pi}}\exp\left(-\frac{1}{2\sigma^2}(x-\mu)^2\right)$, $x \in \R$, for which $\mu \in \R$ is our unknown parameter, whereas $\sigma>0$ is considered to be known. The MLE for $\theta=\mu$ is $$\hat{\theta}_n(\boldsymbol{X}) = \frac{1}{n}\sum_{i=1}^{n}X_i;$$
\item the normal distribution, where now the mean $\mu$ is known and $\theta=\sigma^2$ represents the unknown parameter, with $$\hat{\theta}_n(\boldsymbol{X}) = \frac{1}{n}\sum_{i=1}^{n}(X_i - \mu)^2;$$
\item the Weibull distribution with density $f(x|\alpha,\sigma) = \frac{\alpha}{\sigma}\left(\frac{x}{\sigma}\right)^{\alpha-1}\exp\left(-\left(\frac{x}{\sigma}\right)^\alpha\right)$, $x\geq0$, where $\sigma$ is the unknown scale parameter and $\alpha > 0$ is fixed. The MLE for $\theta=\sigma$ is defined through $$\left(\hat{\theta}_n(\boldsymbol{X})\right)^{\alpha}=\frac{1}{n}\sum_{i=1}^nX_i^{\alpha};$$
\item the Laplace scale model with density $f(x|\sigma) = \frac{1}{2\sigma}\exp(-|x|/\sigma),\theta=\sigma>0,$ over $\R$, for which $$\hat{\theta}_n(\boldsymbol{X})=\frac{1}{n}\sum_{i=1}^n|X_i|.$$
\end{itemize}
Moreover, the broad one-parameter exponential families do satisfy condition~\eqref{MLEfunc}; see Proposition \ref{Propexpofam}  for details. Hence, our results do apply to most of the well-known distributions. 

We now present in detail the notation and general assumptions   made throughout the paper.  We write ${\rm E}_\theta[]$ the expectation under the specific value $\theta$ of the parameter. In line with the notation used above, the joint density or probability mass function of $\Xb$ is written $f(\boldsymbol{x}|\theta)$. The true, unknown value of the parameter is $\theta_0$ and $\Theta$ denotes the parameter space. For $X_i = x_i$  some observed values, the likelihood function is denoted by $L(\theta; \boldsymbol{x}) = f(\boldsymbol{x}|\theta)$ and we denote its natural logarithm, called the log-likelihood function, by $l(\theta;\boldsymbol{x})$.  The derivatives of the log-likelihood function with respect to $\theta$ are $l'(\theta;\boldsymbol{x}),l''(\theta;\boldsymbol{x}),\ldots, l^{(j)}(\theta;\boldsymbol{x})$, for $j$ any integer greater than 2, and $i(\theta)$ denotes the expected Fisher information number for one random variable. Whenever the MLE exists and is also unique, we will write it as before under the form $\hat{\theta}_n(\boldsymbol{X})$. For $\Theta$ being an open interval, we use the results in \cite{Makelainen} to secure the existence and uniqueness of the MLE. Thus, it suffices to assume that:
\begin{itemize}
\item[(A1)] The log-likelihood function $l(\theta;\boldsymbol{x})$ is a twice continuously differentiable function with respect to $\theta$ and the parameter varies in an open interval $(a,b)$, where $a, b \in \mathbb{R}\cup\left\lbrace-\infty,\infty\right\rbrace$ and $a < b$;
\item[(A2)] $\underset{\theta \to a, b}\lim l(\theta;\boldsymbol{x}) = -\infty$;
\item[(A3)] $l''(\theta;\boldsymbol{x}) < 0$ at every point $\theta \in (a,b)$ for which $l'(\theta;\boldsymbol{x}) = 0$.
\end{itemize}
Note that we tacitly assume those conditions in Theorem~\ref{theoMLEnor} when requiring existence and uniqueness of the MLE. Asymptotic normality further requires  the following sufficient regularity conditions:
\begin{itemize}
\item[(R1)] the parameter is identifiable, which means that if $\theta \neq \theta'$, then $\exists x: f(x|\theta)\neq f(x|\theta')$;
\item[(R2)] 
the density $f(x|\theta)$ is three times differentiable with respect to $\theta$, the third derivative is continuous in $\theta$ and $\int f(x|\theta)\,\mathrm{d}x$ can be differentiated three times under the integral sign;
\item[(R3)] for any $\theta_0 \in \Theta$ and for $\mathbb{X}$ denoting the support of $f(x|\theta)$, there exists a positive number $\epsilon$ and a function $M(x)$ (both of which may depend on $\theta_0$) such that 
\begin{equation}
\nonumber \left|\frac{\mathrm{d}^3}{\mathrm{d}\theta^3}\log f(x|\theta)\right| \leq M(x)\quad \forall x \in \mathbb{X},\;\; \theta_0 - \epsilon < \theta < \theta_0 + \epsilon,
\end{equation}
with ${\rm E}_{\theta_0}[M(X)] < \infty$;
\item[(R4)] $i(\theta) > 0, \; \forall \theta \in \Theta$.
\end{itemize}
These conditions, in particular (R2), ensure that, provided the respective expressions exist, ${\rm E}_\theta[l'(\theta;\Xb)]=0$ and ${\rm Var}_\theta[l'(\theta;\Xb)]=ni(\theta)$. These conditions form the basis of Theorem~\ref{theoMLEnor} above; see page 472 of \cite{Casella} for a basic sketch of the proof.

The first paper to address the problem of assessing the accuracy of the asymptotic normal approximation for MLE is \cite{Anastasiou_Reinert}. After deriving general bounds on Zolotarev-type distances, they use the \emph{bounded Wasserstein distance} $d_{bW}$, which is also known as Fortet-Mourier distance (see, e.g., \cite{NP11}) and is linked to the Kolmogorov distance ($H$ is the class of indicator functions of half-spaces) via $d_K(\cdot,\cdot)\leq2\sqrt{d_{bW}(\cdot,\cdot)}$. We state  in Theorem \ref{Theoremnoncan} of Section \ref{sec:generalc} the bound obtained in that paper. For the broad class of distributions satisfying \eqref{MLEfunc}, our bound is better than, or at least as good as, the \cite{Anastasiou_Reinert} bound (hereafter referred to as AR-bound) both in terms of sharpness and simplicity. The tools we use to reach this result are the Delta method, Stein's method for normal approximation, Taylor expansions and conditional expectations. 

The paper is organised as follows. Our new upper bound is described, proved and compared to the AR-bound in Section~\ref{sec:generalc}. In Section~\ref{sec:exo} we then apply our results to the class of one-parameter exponential family distributions and treat some specific examples in detail. 

\section{New bounds on the distance to the normal distribution for the MLE}
\label{sec:generalc}

In order to obtain bounds on the aforementioned  distance, we partly employ the following lemma. From now on, unless otherwise stated, $||\cdot||$ denotes the infinity norm $||\cdot||_{\infty}$.

\begin{lemma}[\cite{Gesinepaper}]
\label{Gesinetheorem}
Let $Y_1, \ldots, Y_n$ be independent random variables with ${\rm E}(Y_i) = 0$, ${\rm Var}(Y_i) = \sigma^2 > 0$ and ${\rm E}\left[\left|Y_i\right|^3\right] < \infty$. Let $W=\frac{1}{\sqrt{n}}\sum_{i=1}^{n} Y_i$, with ${\rm E}(W) = 0$, ${\rm Var}(W) = \sigma^2$ and let $K \sim \mathcal{N}(0,\sigma^2)$. Then for any function $h \in H$, with $H$ given in \eqref{classfunctions}, one has
\begin{equation*}
\left|{\rm E}[h(W)] - {\rm E}[h(K)]\right| \leq \frac{\|h'\|}{\sqrt{n}}\left(2 + \frac{1}{\sigma^3}{\rm E}\left[\left|Y_1\right|^3\right]\right).
\end{equation*}
\end{lemma}
As we shall see below, our strategy consists in benefiting from the special form of $q(\hat\theta_n(\Xb))$, which is a sum of random variables and thus allows us to use the sharp bound of this lemma. It is precisely at this point that the Delta method comes into play: abusing notations and language, instead of comparing $\hat\theta_n(\Xb)$ to $Z\sim \mathcal{N}(0,1)$ we rather compare $q(\hat\theta_n(\Xb))$ to $Z$, and then bound the distance between $\hat\theta_n(\Xb)$ and $q(\hat\theta_n(\Xb))$. The outcome of this approach is the next theorem, the main result of the present paper.
\vspace{0.001in} 
\begin{theorem}
\label{Theoremdelta}
Let $X_1, \ldots, X_n$ be i.i.d. random variables with probability density (or mass) function $f(x_i|\theta)$ and let $Z \sim \mathcal{N}(0,1)$. Assume that (A1)-(A3) and the regularity conditions (R1)-(R4) are satisfied, and hence the MLE $\hat{\theta}_n(\boldsymbol{X})$ exists and is unique. Furthermore let $q:\Theta \rightarrow \mathbb{R}$ be a one-to-one twice differentiable function with $q'(\theta) \neq 0\,\, \forall \theta \in \Theta$ and such that $q\left(\hat{\theta}_n(\boldsymbol{X})\right) = \frac{1}{n}\sum_{i=1}^{n}g(X_i)$, where the mapping $g:\mathbb{R}\rightarrow\mathbb{R}$ is such that ${\rm E}\left[\left|g(X_1)-q(\theta_0)\right|^3\right]<\infty$ for $\theta_0\in\Theta$ the true value of the parameter. Also, there exists a positive constant $0 < \epsilon = \epsilon(\theta_0)$ as in (R3) with $(\theta_0 - \epsilon, \theta_0 + \epsilon)\subset \Theta$. Then, for any $h\in H$ as in \eqref{classfunctions} we have
\begin{align}
\label{bounddeltageneral}
\nonumber &\left|{\rm E}\left[h\left(\sqrt{n\,i(\theta_0)}\left(\hat{\theta}_n(\boldsymbol{X})- \theta_0\right)\right)\right] - {\rm E}\left[h(Z)\right]\right| \\
\nonumber & \leq \frac{\|h'\|}{\sqrt{n}}\left(2 + \frac{[i(\theta_0)]^{\frac{3}{2}}}{\left|q'(\theta_0)\right|^3}{\rm E}\left[\left|g(X_1) - q(\theta_0)\right|^3\right]\right)\\
&+ {\rm E}\left[\left(\hat{\theta}_n(\boldsymbol{X})-\theta_0\right)^2\right]\left(2\frac{\|h\|}{\epsilon^2}\mathbbm{1}\left\lbrace \exists\theta\in\Theta:\,q(\theta)\neq\theta\right\rbrace + \frac{\|h'\|\sqrt{n\,i(\theta_0)}}{2\left|q'(\theta_0)\right|}\sup_{\theta:|\theta-\theta_0|\leq\epsilon}\left|q''(\theta)\right|\right).
\end{align}
\end{theorem}\vspace{2mm}

\begin{proof} The asymptotic normality of the MLE is explicitly stated in Theorem~\ref{theoMLEnor}. Applying the widely known Delta method to this result in combination with the requirement $q'(\theta_0) \neq 0$ we obtain
\begin{equation}
\label{Delta_method}
\frac{\sqrt{n\:i(\theta_0)}}{q'(\theta_0)}\left(q\left(\hat{\theta}_n(\boldsymbol{X})\right)-q\left(\theta_0\right)\right)\xrightarrow[{n \to \infty}]{{\rm d}} \mathcal{N}(0, 1),
\end{equation}
with $q\left(\hat{\theta}_n(\boldsymbol{X})\right) = \frac{1}{n}\sum_{i=1}^{n}g(X_i)$. Using the triangle inequality we get that
\begin{align}
\nonumber &\left|{\rm E}\left[ h\left(\sqrt{n\:i(\theta_0)}\left(\hat{\theta}_n(\boldsymbol{X}) - \theta_0\right) \right) \right] - {\rm E}[h(Z)]\right|  \\
\label{Deltaterm1}&\leq\left|{\rm E}\left[ h\left(\frac{\sqrt{n\:i(\theta_0)}}{q'(\theta_0)}\left(q\left(\hat{\theta}_n(\boldsymbol{X})\right) - q\left(\theta_0\right)\right) \right) \right] - {\rm E}[h(Z)]\right|\\
\label{Deltaremainder}& + \left|{\rm E}\left[h\left(\sqrt{n\:i(\theta_0)}\left(\hat{\theta}_n(\boldsymbol{X}) -\theta_0\right)\right) - h\left(\frac{\sqrt{n\:i(\theta_0)}}{q'(\theta_0)}\left(q\left(\hat{\theta}_n(\boldsymbol{X})\right) -q\left(\theta_0\right)\right)\right)\right]\right|.
\end{align}
We first start to obtain an upper bound for \eqref{Deltaterm1} using (indirectly) Stein's method via Lemma~\ref{Gesinetheorem}. Some simple rewriting yields
\begin{equation*}
\begin{aligned}
&\frac{\sqrt{n\:i(\theta_0)}}{q'(\theta_0)}\left(q\left(\hat{\theta}_n(\boldsymbol{X})\right) -q\left(\theta_0\right)\right) = \frac{\sqrt{n\:i(\theta_0)}}{q'\left(\theta_0\right)}\left(\frac{1}{n}\sum_{i=1}^{n}g\left(X_i\right)-q\left(\theta_0\right)\right)\\
&= \frac{1}{\sqrt{n}}\sum_{i=1}^{n}\left\lbrace\frac{\sqrt{i\left(\theta_0\right)}}{q'(\theta_0)}\left(g(X_i)-q(\theta_0)\right)\right\rbrace = \frac{1}{\sqrt{n}}\sum_{i=1}^{n}Y_i,
\end{aligned}
\end{equation*}
where  $Y_i =  \frac{\sqrt{i(\theta_0)}}{q'(\theta_0)}\left(g(X_i) - q(\theta_0)\right),i=1,2,\ldots,n$ and, obviously, the $Y_i$'s are independent and identically distributed random variables. The Central Limit Theorem applied to $\frac{1}{\sqrt{n}}\sum_{i=1}^{n}Y_i$ implies $\frac{1}{\sqrt{n}}\sum_{i=1}^{n}(Y_i-{\rm E}(Y_1))\stackrel{{\rm d}}{\rightarrow}\mathcal{N}(0,{\rm Var}(Y_1))$. From \eqref{Delta_method} we know however that $\frac{1}{\sqrt{n}}\sum_{i=1}^{n}Y_i\stackrel{{\rm d}}{\rightarrow}\mathcal{N}(0,1)$; comparing the two asymptotic results reveals that, necessarily (since two normal distributions can only be equal if their expectations and variances are the same), we have
\begin{equation*}
{\rm E}\left[g(X_1)\right] = q(\theta_0)\quad\quad \mbox{and}\quad\quad {\rm Var}\left[g(X_1)\right] = \frac{(q'(\theta_0))^2}{i(\theta_0)},
\end{equation*}
%
where condition (R4) allows to divide by $i(\theta_0)$. Hence, ${\rm E}[Y_i] = 0$ and ${\rm Var}[Y_i] =  1$ (as Lemma \ref{Gesinetheorem} requires). Applying the result of the lemma we get \small
\begin{align}
\label{boundStein}
\nonumber &\left|{\rm E}\left[ h\left(\frac{\sqrt{n\:i(\theta_0)}}{q'(\theta_0)}\left(q\left(\hat{\theta}_n(\boldsymbol{X})\right) - q\left(\theta_0\right)\right) \right) \right] - {\rm E}[h(Z)]\right|\\
& \leq \frac{\|h'\|}{\sqrt{n}}\left(2 + \frac{[i(\theta_0)]^{\frac{3}{2}}}{\left|q'(\theta_0)\right|^3}{\rm E}\left[\left|g(X_1) - q(\theta_0)\right|^3\right]\right).
\end{align}\normalsize
Now we are searching for an upper bound on \eqref{Deltaremainder}. Since the case $q(\theta)=\theta$ is obvious, we from here on assume that $q(\theta)\neq\theta$.  To do so, we denote by
\begin{align}
\nonumber A:= A(q,\theta_0,\boldsymbol{X}):=& h\left(\sqrt{n\:i(\theta_0)}\left(\hat{\theta}_n(\boldsymbol{X}) -\theta_0\right)\right)\\
\nonumber &\;\; - h\left(\frac{\sqrt{n\:i(\theta_0)}}{q'(\theta_0)}\left(q\left(\hat{\theta}_n(\boldsymbol{X})\right) -q\left(\theta_0\right)\right)\right)
\end{align}
and our scope is to find an upper bound for $\left|{\rm E}\left[A\right]\right|$. Using the law of total expectation related to conditioning on $\left|\hat{\theta}_n(\boldsymbol{X}) - \theta_0\right| > \epsilon$ or $\left|\hat{\theta}_n(\boldsymbol{X}) - \theta_0\right| \leq \epsilon$ and the triangle inequality we obtain
\begin{align}
\nonumber \left|{\rm E}[A]\right| &= \left|{\rm E}\left[A\,\,\middle|\,\,\left|\hat{\theta}_n(\boldsymbol{X}) - \theta_0\right|>\epsilon\right]\Prob\left(\left|\hat{\theta}_n(\boldsymbol{X})-\theta_0\right|>\epsilon\right)\right.\\
\nonumber &\left.\;\;\;\;\;+ {\rm E}\left[A\,\,\middle|\,\,\left|\hat{\theta}_n(\boldsymbol{X}) - \theta_0\right|\leq\epsilon\right]\Prob\left(\left|\hat{\theta}_n(\boldsymbol{X})-\theta_0\right|\leq\epsilon\right)\right|\\
\nonumber & \leq \left|{\rm E}\left[A\,\,\middle|\,\,\left|\hat{\theta}_n(\boldsymbol{X}) - \theta_0\right|>\epsilon\right]\right|\Prob\left(\left|\hat{\theta}_n(\boldsymbol{X})-\theta_0\right|>\epsilon\right)\\
\nonumber &\;\;\;\;\;+ \left|{\rm E}\left[A\,\,\middle|\,\,\left|\hat{\theta}_n(\boldsymbol{X}] - \theta_0\right|\leq\epsilon\right]\right|\Prob\left(\left|\hat{\theta}_n(\boldsymbol{X})-\theta_0\right|\leq\epsilon\right)\\
\nonumber  & \leq {\rm E}\left[\left|A\right|\,\,\middle|\,\,\left|\hat{\theta}_n(\boldsymbol{X}) - \theta_0\right|>\epsilon\right]\Prob\left(\left|\hat{\theta}_n(\boldsymbol{X})-\theta_0\right|>\epsilon\right)\\
\nonumber &\;\;+ {\rm E}\left[\left|A\right|\,\,\middle|\,\,\left|\hat{\theta}_n(\boldsymbol{X}) - \theta_0\right|\leq\epsilon\right]\Prob\left(\left|\hat{\theta}_n(\boldsymbol{X})-\theta_0\right|\leq\epsilon\right).
\end{align}
Markov's inequality and the elementary results of $\Prob\left(\left|\hat{\theta}_n(\boldsymbol{X})-\theta_0\right|\leq\epsilon\right) \leq 1$ and  $|A|\leq 2\|h\|$ further yield 
\begin{equation}
\label{boundforA}
 |{\rm E}[A]| \leq 2\|h\|\frac{{\rm E}\left[\left(\hat{\theta}_n(\boldsymbol{X}) - \theta_0\right)^2\right]}{\epsilon^2} + {\rm E}\left[|A|\,\,\middle| \,\,\left|\hat{\theta}_n(\boldsymbol{X}) - \theta_0\right|\leq\epsilon\right].
\end{equation}
We now focus on the conditional expectation on the right-hand side of \eqref{boundforA}. A second-order Taylor expansion of $q\left(\hat{\theta}_n(\boldsymbol{x})\right)$ about $\theta_0$ gives
\begin{equation}
\label{Taylor_for_q}
q\left(\hat{\theta}_n(\boldsymbol{x})\right) = q(\theta_0) + \left(\hat{\theta}_n(\boldsymbol{x}) - \theta_0\right)q'(\theta_0) + \frac{1}{2}\left(\hat{\theta}_n(\boldsymbol{x}) - \theta_0\right)^2q''(\theta^*),
\end{equation}
for $\theta^*$ between $\hat{\theta}_n(\boldsymbol{x})$ and $\theta_0$. Since we assume that $q'(\theta)\neq 0\,\, \forall \theta \in \Theta$, we can multiply both sides in \eqref{Taylor_for_q} with $\frac{\sqrt{n\,i(\theta_0)}}{q'(\theta_0)}$. Rearranging the terms, we obtain
\begin{equation}
\nonumber \frac{\sqrt{n\,i(\theta_0)}\left(q\left(\hat{\theta}_n(\boldsymbol{x})\right)-q(\theta_0)\right)}{q'(\theta_0)} = \sqrt{n\,i(\theta_0)}\left(\hat{\theta}_n(\boldsymbol{x})-\theta_0\right) + \frac{\sqrt{n\,i(\theta_0)}}{2q'(\theta_0)}q''\left(\theta^*\right)\left(\hat{\theta}_n(\boldsymbol{x})-\theta_0\right)^2.
\end{equation}
Using the above result along with another first-order Taylor expansion (recall that\linebreak $\sqrt{n}\left(\hat{\theta}_n(\boldsymbol{X})-\theta_0\right)^2=o_{\rm P}(1)$  as $n\rightarrow\infty$) 
gives
\begin{equation}
\label{Taylor_for_h}
\begin{aligned}
h\left(\sqrt{n\,i(\theta_0)}\left(\hat{\theta}_n(\boldsymbol{x})-\theta_0\right)\right) &- h\left(\frac{\sqrt{n\,i(\theta_0)}}{q'(\theta_0)}\left(q\left(\hat{\theta}_n(\boldsymbol{x})\right)-q(\theta_0)\right)\right)\\
&\;=- \frac{\sqrt{n\,i(\theta_0)}}{2q'(\theta_0)}q''(\theta^*)\left(\hat{\theta}_n(\boldsymbol{x})-\theta_0\right)^2h'(t(\boldsymbol{x})),
\end{aligned}
\end{equation}
where $t(\boldsymbol{x})$ is between $\sqrt{n\,i(\theta_0)}\left(\hat{\theta}_n(\boldsymbol{x})-\theta_0\right)$ and $\frac{\sqrt{n\,i(\theta_0)}}{q'(\theta_0)}\left(q\left(\hat{\theta}_n(\boldsymbol{x})\right)-q(\theta_0)\right)$. Equality~\eqref{Taylor_for_h} combined with Lemma 2.1 in \cite{Anastasiou_Reinert} related to conditional expectations yields
\begin{align}
\label{boundremainder}
&\nonumber {\rm E}\left[|A|\,\middle|\,\left|\hat{\theta}_n(\boldsymbol{X})-\theta_0\right|\leq\epsilon\right]\\
&\nonumber = {\rm E}\left[\left|-\frac{\sqrt{n\,i(\theta_0)}}{2q'(\theta_0)}q''(\theta^*)\left(\hat{\theta}_n(\boldsymbol{X})-\theta_0\right)^2h'(t(\boldsymbol{X}))\right|\,\middle|\,\left|\hat{\theta}_n(\boldsymbol{X})-\theta_0\right|\leq\epsilon\right]\\
\nonumber & \leq \frac{\|h'\|\sqrt{n\,i(\theta_0)}}{2\left|q'(\theta_0)\right|}{\rm E}\left[\left|q''(\theta^*)\right|\left(\hat{\theta}_n(\boldsymbol{X})-\theta_0\right)^2\,\,\middle|\,\,\left|\hat{\theta}_n(\boldsymbol{X})-\theta_0\right|\leq\epsilon\right]\\
\nonumber& \leq \frac{\|h'\|\sqrt{n\,i(\theta_0)}}{2\left|q'(\theta_0)\right|}\sup_{\theta:|\theta-\theta_0|\leq\epsilon}\left|q''(\theta)\right|{\rm E}\left[\left(\hat{\theta}_n(\boldsymbol{X})-\theta_0\right)^2\,\,\middle|\,\,\left|\hat{\theta}_n(\boldsymbol{X})-\theta_0\right|\leq\epsilon\right]\\
 & \leq \frac{\|h'\|\sqrt{n\,i(\theta_0)}}{2\left|q'(\theta_0)\right|}\sup_{\theta:|\theta-\theta_0|\leq\epsilon}\left|q''(\theta)\right|{\rm E}\left[\left(\hat{\theta}_n(\boldsymbol{X})-\theta_0\right)^2\right].
\end{align}
Combining the bounds in \eqref{boundStein}, \eqref{boundforA} and \eqref{boundremainder} gives the result of the theorem.
\end{proof}
\begin{remark} 
\textbf{(1)} The convergence of the second and third terms is governed by the asymptotic behaviour of ${\rm E}\left[\left(\hat{\theta}_n(\boldsymbol{X})-\theta_0\right)^2\right]$, whose rate of convergence is $\mathcal{O}\left(\frac{1}{n}\right)$. This result is obtained using the decomposition
\begin{equation}
\label{MSE}
{\rm E}\left[(\hat{\theta}_n(\boldsymbol{X}) - \theta_0)^2\right] = {\rm Var}[\hat{\theta}_n(\boldsymbol{X})] + {\rm bias}^2[\hat{\theta}_n(\boldsymbol{X})].
\end{equation}
Under the standard asymptotics (from the regularity conditions (R1)-(R4)) the MLE is asymptotically efficient, meaning that
\begin{equation}
\nonumber n{\rm Var}[\hat{\theta}_n(\boldsymbol{X})] \xrightarrow[{n \to \infty}]{{}} [i(\theta_0)]^{-1},
\end{equation}
and hence the variance of the MLE is of order $\frac{1}{n}$. In addition, from Theorem \ref{theoMLEnor} the bias of the MLE is of order $\frac{1}{\sqrt{n}}$; see also \cite{Cox}, where no explicit conditions are given. Combining these two results and using \eqref{MSE} shows that the mean squared error of the MLE is of order $\frac{1}{n}$.\\
\textbf{(2)} In the simplest possible situation where $\hat\theta_n(\Xb)$ is already a sum of \mbox{i.i.d.} terms, $q(x)=x$ and hence our upper bound simplifies to 
$$
\left|{\rm E}\left[h\left(\sqrt{n\,i(\theta_0)}\left(\hat{\theta}_n(\boldsymbol{X})- \theta_0\right)\right)\right] - {\rm E}\left[h(Z)\right]\right| \leq \frac{\|h'\|}{\sqrt{n}}\left(2 + [i(\theta_0)]^{\frac{3}{2}}{\rm E}\left[\left|g(X_1) - \theta_0\right|^3\right]\right),
$$
which is equivalent to Lemma \ref{Gesinetheorem}.
\end{remark}

In order to appreciate the sharpness and simplicity of our bound~\eqref{bounddeltageneral}, we  compare it to the AR-bound. To this end, we now state the main result of \cite{Anastasiou_Reinert}. 
\begin{theorem}[\cite{Anastasiou_Reinert}]
\label{Theoremnoncan}
Let $X_1, X_2, \ldots, X_n$ be i.i.d. random variables with density or frequency function $f(x_i|\theta)$ such that the regularity conditions (R1)-(R4) are satisfied and that the MLE, $\hat{\theta}_n(\boldsymbol{X})$, exists and it is unique. Assume that ${\rm E}\left[\left|\left(\frac{\mathrm{d}}{\mathrm{d}\theta}{\rm log}f(X_1|\theta)\right)_{\theta=\theta_0}\right|^3\right] < \infty$ and that ${\rm E}\left[\left(\hat{\theta}_n(\boldsymbol{X}) - \theta_0\right)^4\right] < \infty$. Let $0 < \epsilon = \epsilon(\theta_0)$ be such that $(\theta_0 - \epsilon, \theta_0 + \epsilon) \subset \Theta$ as in (R3) and let $Z \sim {\mathcal N}(0,1)$. Then for any function $h$ that is absolutely continuous and bounded,
\begin{align*}
\nonumber & \left|{\rm E}\left[h\left(\sqrt{n\,i(\theta_0)}\left(\hat{\theta}_n(\boldsymbol{X})- \theta_0\right)\right)\right] - {\rm E}\left[h(Z)\right]\right|\\
& \leq \frac{\|h'\|}{\sqrt{n}}\left(2 + \frac{1}{[i(\theta_0)]^{\frac{3}{2}}}{\rm E}\left[\left|\left(\frac{\mathrm{d}}{\mathrm{d}\theta}{\rm log}f(X_1|\theta)\right)_{\theta=\theta_0}\right|^3\right]\right)\\
&+ 2\|h\|\frac{{\rm E}\left[\left(\hat{\theta}_n(\boldsymbol{X}) - \theta_0\right)^2\right]}{\epsilon^2} + \frac{\|h'\|}{\sqrt{n\:i(\theta_0)}}\left\lbrace\vphantom{(\left(\sup_{\theta:|\theta-\theta_0|\leq\epsilon}\left|l^{(3)}(\theta;\boldsymbol{X})\right|\right)^2}{\rm E}\left[\left|R_2(\theta_0;\boldsymbol{X})\right| \middle | |\hat{\theta}_n(\boldsymbol{X}) - \theta_0| \leq \epsilon\right]\right.\\
\nonumber & + \left. \frac{1}{2}\left[{\rm E}\left[\left(\sup_{\theta:|\theta-\theta_0|\leq\epsilon}\left|l^{(3)}(\theta;\boldsymbol{X})\right|\right)^2\middle| |\hat{\theta}_n(\boldsymbol{X}) - \theta_0|\leq\epsilon\right]\right]^{\frac{1}{2}}\left[{\rm E}\left[\left(\hat{\theta}_n(\boldsymbol{X}) - \theta_0\right)^4\right]\right]^{\frac{1}{2}}\vphantom{(\left(\sup_{\theta:|\theta-\theta_0|\leq\epsilon}\left|l^{(3)}(\theta;\boldsymbol{X})\right|\right)^2}\right\rbrace,
\end{align*}
where
\begin{equation}
\label{R2eq}
R_2(\theta_0;\boldsymbol{x}) = (\hat{\theta}_n(\boldsymbol{x}) - \theta_0)\left(l''(\theta_0;\boldsymbol{x}) + n\:i(\theta_0)\right).
\end{equation}
\end{theorem}
Obvious observations are  that the AR-bound requires finiteness of the fourth moment of $\hat\theta_n(\Xb)-\theta_0$ and that this bound is more complicated than ours. Let us now comment on the bounds term by term.
\begin{itemize}
\item In the first term of the bounds, the different positioning of the expected Fisher information number is explained by the fact that  we apply Lemma \ref{Gesinetheorem} to the standardised version of $g(X_1), g(X_2), \ldots, g(X_n)$, which have variance $\frac{[q'(\theta_0)]^2}{i(\theta_0)}$, while \cite{Anastasiou_Reinert} obtain the result by applying the lemma after standardising $l'(\theta_0;X_1), l'(\theta_0;X_2), \ldots, l'(\theta_0;X_n)$, which have variance  equal to $i(\theta_0)$.
\item The second and  third terms vanish in our bound when $q(\theta) = \theta\,\forall \theta\in\Theta$, while the AR-bound does not take this simplification into account. In addition, when $q(\theta)\neq \theta$ the second term is the same in both bounds, whereas the third term in our bound reads 
\begin{equation}\label{ourterm3}
{\rm E}\left[\left(\hat{\theta}_n(\boldsymbol{X})-\theta_0\right)^2\right] \frac{\|h'\|\sqrt{n\,i(\theta_0)}}{2\left|q'(\theta_0)\right|}\sup_{\theta:|\theta-\theta_0|\leq\epsilon}\left|q''(\theta)\right|
\end{equation}
 and is to be compared to 
\begin{eqnarray}
&&\hspace{-9mm}\frac{\|h'\|}{2\sqrt{ni(\theta_0)}}\left[{\rm E}\left[\left(\sup_{\theta:|\theta-\theta_0|\leq\epsilon}\left|l^{(3)}(\theta;\boldsymbol{X})\right|\right)^2\middle| |\hat{\theta}_n(\boldsymbol{X}) - \theta_0|\leq\epsilon\right]\right]^{\frac{1}{2}}\left[{\rm E}\left[\left(\hat{\theta}_n(\boldsymbol{X}) - \theta_0\right)^4\right]\right]^{\frac{1}{2}}\nonumber\\
&& \hspace{-9mm}+\frac{\|h'\|}{\sqrt{n\:i(\theta_0)}}\vphantom{(\left(\sup_{\theta:|\theta-\theta_0|\leq\epsilon}\left|l^{(3)}(\theta;\boldsymbol{X})\right|\right)^2}{\rm E}\left[\left|R_2(\theta_0;\boldsymbol{X})\right| \middle | |\hat{\theta}_n(\boldsymbol{X}) - \theta_0| \leq \epsilon\right]\label{ARterm3}
\end{eqnarray}
where $R_2(\theta_0;\boldsymbol{x}) = (\hat{\theta}_n(\boldsymbol{x}) - \theta_0)\left(l''(\theta_0;\boldsymbol{x}) + n\:i(\theta_0)\right)$. The second derivative, $q''(\theta)$, plays in our bound the role of $l^{(3)}(\theta;\Xb)$, up to an important difference: $l^{(3)}(\theta;\Xb)$ is a sum. Consequently, the first term in \eqref{ARterm3} has $\sqrt{n}$ in its numerator, exactly as in~\eqref{ourterm3}. The distinct positioning of the information quantity $i(\theta_0)$ has the same reason as explained above. Besides the obvious additional term in the AR bound (the second term in~\eqref{ARterm3}),  our bound is also clearly sharper at the level of moments of $\hat\theta_n(\Xb)-\theta_0$ since 
$${\rm E}\left[\left(\hat{\theta}_n(\boldsymbol{X})-\theta_0\right)^2\right]\leq\left[{\rm E}\left[\left(\hat{\theta}_n(\boldsymbol{X}) - \theta_0\right)^4\right]\right]^{\frac{1}{2}}$$
by the Cauchy-Schwarz inequality.
\end{itemize}
From this comparison one sees that our new bound is simpler and, moreover, has one term less. This is particularly striking in the simplest possible setting where $\hat\theta_n(\Xb)$ is a sum of \mbox{i.i.d.} terms, where our bound clearly improves on the AR-bound. An advantage of the AR-bound is its wider applicability as it works for all MLE settings, even when an analytic expression of the MLE is not known. 
\section{Calculation of the bound in different scenarios}
\label{sec:exo}
In this section we shall consider different examples for which we explicitly calculate our upper bound from Theorem~\ref{Theoremdelta} and compare it to the AR-bound. To further assess its accuracy, we simulate data from various distributions and compare our bound to the actual distance between the unknown exact law of the MLE and its asymptotic normal law, for distinct values of the sample size $n$.
\subsection{Bounds for one-parameter exponential families}
\label{sec:one_parameter}
The probability density (or mass) function for one-parameter exponential families  is given by
\begin{equation}
\label{density_exp_family}
f(x|\theta) = {\rm exp}\left\lbrace k(\theta)T(x) - A(\theta) + S(x)\right\rbrace\mathbbm{1}_{\{x \in B\}},
\end{equation}
where the set $B = \left\lbrace x:f(x|\theta)>0 \right\rbrace$ is the support of the density and does not depend on~$\theta$; $k(\theta)$ and $A(\theta)$ are functions of the parameter; $T(x)$ and $S(x)$ are functions only of the data. Whenever $k(\theta) = \theta$ we have the so-called \emph{canonical case}, where $\theta$ and $T(X)$ are called the \textit{natural} \textit{parameter} and \textit{natural} \textit{observation} \citep{Casella}. The identifiability constraint in (R1) entails that $k'(\theta)\neq 0$ \citep{Geyer}, an important detail for the following investigation.
\begin{proposition}\label{Propexpofam}
Suppose $X_1, \ldots, X_n$ are i.i.d. with probability density (or mass) function that can be expressed in the form of \eqref{density_exp_family}. Assume that (A1)-(A3) and the regularity conditions (R1)-(R4) are satisfied, and hence the MLE $\hat{\theta}_n(\boldsymbol{X})$ exists and is unique.  Let $\theta\mapsto D(\theta) = \frac{A'(\theta)}{k'(\theta)}$ be invertible. Then $q(\cdot) = D(\cdot)$, with $q:\Theta \rightarrow \mathbb{R}$ as in Theorem \ref{Theoremdelta}.
\end{proposition}
\begin{proof}
Using \eqref{density_exp_family}, we have that
\begin{align}
\nonumber &L(\theta;\boldsymbol{x}) = \prod_{i=1}^{n} f(x_i|\theta) = {\rm exp}\left\lbrace k(\theta)\sum_{i=1}^{n}T(x_i) - n\,A(\theta) +\sum_{i=1}^{n}S(x_i)\right\rbrace,\\
\nonumber &l(\theta;\boldsymbol{x}) = k(\theta)\sum_{i=1}^{n}T(x_i) - n\,A(\theta) + \sum_{i=1}^{n}S(x_i),
\end{align}
and hence 
$$l'(\theta;\boldsymbol{x}) = k'(\theta)\sum_{i=1}^{n}T(x_i) - n\,A'(\theta) = 0 \Longleftrightarrow D(\theta) = \frac{1}{n}\sum_{i=1}^{n}T(x_i),
$$
which means that $\hat{\theta}_n(\boldsymbol{X}) = D^{-1}\left(\frac{1}{n}\sum_{i=1}^{n} T(X_i)\right)$ under the invertibility assumption for~$D(\theta)$. The claim readily follows. 
\end{proof}
This result hence shows that, as announced in the Introduction, the broad one-parameter exponential families do satisfy~\eqref{MLEfunc}. Consequently, Theorem~\ref{Theoremdelta} can be applied to~\eqref{density_exp_family}, resulting in

\begin{corollary}
\label{Theoremdeltaexp}
Let $X_1, \ldots, X_n$ be i.i.d. random variables with the probability density (or mass) function of a single-parameter exponential family. Assume that (A1)-(A3) and (R1)-(R4) are satisfied, and hence the MLE $\hat{\theta}_n(\boldsymbol{X})$ exists and is unique. With $Z \sim \mathcal{N}(0,1)$, $h \in H$ as defined in \eqref{classfunctions} and  $0 < \epsilon = \epsilon(\theta_0):(\theta_0 - \epsilon, \theta_0 + \epsilon)\subset \Theta$ as in (R3), we obtain
\begin{align*}
\nonumber &\left|{\rm E}\left[h\left(\sqrt{n\,i(\theta_0)}\left(\hat{\theta}_n(\boldsymbol{X})- \theta_0\right)\right)\right] - {\rm E}\left[h(Z)\right]\right|\\
& \leq \frac{\|h'\|}{\sqrt{n}}\left(2 + \frac{|k'(\theta_0)|^{3}{\rm E}\left[\left|T(X_1) - D(\theta_0)\right|^3\right]}{\left|A''(\theta_0)-k''(\theta_0)D(\theta_0)\right|^{\frac{3}{2}}}\right)\\
\nonumber &+ {\rm E}\left[\left(\hat{\theta}_n(\boldsymbol{X})-\theta_0\right)^2\right]\left(\vphantom{(\left(\sup_{\theta:|\theta-\theta_0|\leq\epsilon}\left|l^{(3)}(\theta;\boldsymbol{X})\right|\right)^2}2\frac{\|h\|}{\epsilon^2}\mathbbm{1}\left\lbrace \exists\theta\in\Theta:D(\theta)\neq\theta\right\rbrace\right.\\
\nonumber &\left.\qquad\qquad\qquad\qquad\qquad\quad +\frac{\|h'\|\sqrt{n}|k'(\theta_0)|}{2\sqrt{|A''(\theta_0)-k''(\theta_0)D(\theta_0)|}}\sup_{\theta:|\theta-\theta_0|\leq\epsilon}\left|D''(\theta)\right|\vphantom{(\left(\sup_{\theta:|\theta-\theta_0|\leq\epsilon}\left|l^{(3)}(\theta;\boldsymbol{X})\right|\right)^2}\right).
\end{align*}
\end{corollary}
\begin{proof}  We readily have 
$$
 i(\theta_0) = {\rm E}\left[-l''(\theta_0;X_1)\right] = A''(\theta_0) - k''(\theta_0){\rm E}[T(X_1)] = \frac{A''(\theta_0)k'(\theta_0) - k''(\theta_0)A'(\theta_0)}{k'(\theta_0)}$$
 and
$q'(\theta_0) = \frac{A''(\theta_0)k'(\theta_0) - k''(\theta_0)A'(\theta_0)}{\left[k'(\theta_0)\right]^2}$. Combining these two results,
\begin{equation}
\nonumber \frac{\sqrt{i(\theta_0)}}{|q'(\theta_0)|} = \frac{|k'(\theta_0)|^{\frac{3}{2}}}{\sqrt{|A''(\theta_0)k'(\theta_0) - k''(\theta_0)A'(\theta_0)|}} = \frac{|k'(\theta_0)|}{\sqrt{|A''(\theta_0)-k''(\theta_0)D(\theta_0)|}}.
\end{equation}
This result, along with the fact that $g(x) = T(x)$ and $q(\theta) = D(\theta)$ by Proposition~\ref{Propexpofam}, allows to deduce the announced upper bound from Theorem~\ref{Theoremdelta}.
\end{proof}
\begin{remark}
\label{remark_canonical_exponential}
It is particularly interesting to spell out this bound in the canonical case $k(\theta)=\theta$. Since then $k'(\theta)=1$, $D(\theta)=A'(\theta)$, we find
\begin{align*}
\nonumber &\left|{\rm E}\left[h\left(\sqrt{n\,i(\theta_0)}\left(\hat{\theta}_n(\boldsymbol{X})- \theta_0\right)\right)\right] - {\rm E}\left[h(Z)\right]\right| \leq \frac{\|h'\|}{\sqrt{n}}\left(2 + \frac{{\rm E}\left[\left|T(X_1) - A'(\theta_0)\right|^3\right]}{\left|A''(\theta_0)\right|^{\frac{3}{2}}}\right)\\
\nonumber &\;\; + {\rm E}\left[\left(\hat{\theta}_n(\boldsymbol{X})-\theta_0\right)^2\right]\left(2\frac{\|h\|}{\epsilon^2}\mathbbm{1}\left\lbrace \exists\theta\in\Theta:A'(\theta)\neq\theta\right\rbrace\right.\\
\nonumber & \left. \qquad\qquad\qquad\qquad\qquad\qquad + \frac{\|h'\|\sqrt{n}}{2\sqrt{|A''(\theta_0)|}}\sup_{\theta:|\theta-\theta_0|\leq\epsilon}\left|A'''(\theta)\right|\right).
\end{align*}
As $i(\theta)=A''(\theta)$ and $l''(\theta;\Xb)=-nA''(\theta)$, $R_2(\theta;\xb)=0$ and straightforward manipulations show that all terms in the AR-bound coincide with those in our bound, except for $\left[{\rm E}\left[\left(\hat{\theta}_n(\boldsymbol{X}) - \theta_0\right)^4\right]\right]^{\frac{1}{2}}$, making the AR-bound less sharp than ours. However, \cite{Anastasiou_Reinert} have shown that, in the canonical exponential setting, their bound can actually have an ${\rm E}\left[\left(\hat{\theta}_n(\boldsymbol{X}) - \theta_0\right)^2\right]$ factor, implying that both bounds are exactly the same in the canonical case. In order to get an idea of how our bound improves on the AR-bound in non-canonical cases, we treat the exponential distribution under a non-canonical parametrisation in Subsection \ref{sec:exoneg}.
\end{remark}
\subsection{Bounds for the Generalised Gamma distribution}
Let us consider $X_1, \ldots, X_n$  i.i.d. random variables from the Generalized Gamma GG($\theta,d,p$) distribution, where the shape parameters $d, p >0$ are considered to be known and the scale parameter $\theta$ is the unknown parameter of interest. The Generalised Gamma distribution includes many other known distributions as special cases:  the Weibull for $d=p$, the Gamma when $p=1$, and the negative exponential when $p=d=1$. Indeed, with $\Gamma(\cdot)$ denoting the Gamma function, the probability density function for $x>0$ is
\begin{equation*}
\begin{aligned}
f(x|\theta) &= \frac{px^{d-1}}{\theta^d}\frac{\exp\left\lbrace -\left(\frac{x}{\theta}\right)^p \right\rbrace}{\Gamma\left(\frac{d}{p}\right)}\\
&= \exp\left\lbrace -\frac{x^p}{\theta^p} + \log p -d\log\theta +(d-1)\log x - \log\left(\Gamma\left(\frac{d}{p}\right)\right) \right\rbrace
\end{aligned}
\end{equation*}
where, in the terminology of one-parameter exponential families, $B = (0,\infty)$, $ \Theta = (0,\infty)$, $T(x) = x^p$, $k(\theta)=-\frac{1}{\theta^p}$, $A(\theta) = d\log{\theta}$ and $S(x)=\log p +(d-1)\log x - \log\left(\Gamma\left(\frac{d}{p}\right)\right)$. Simple steps yield
\begin{align}
\nonumber & l(\theta_0;\boldsymbol{x}) = -\frac{1}{\theta_0^p}\sum_{i=1}^{n}x_i^p + n\log p - nd\log\theta_0 +(d-1)\log\left(\prod_{i=1}^{n}x_i\right) - n\log\left(\Gamma\left(\frac{d}{p}\right)\right)\\
\nonumber & l'(\theta_0;\boldsymbol{x}) = \frac{p}{\theta_0^{p+1}}\sum_{i=1}^{n}x_i^p -n\frac{d}{\theta_0} = 0 \Leftrightarrow  \hat{\theta}_n(\boldsymbol{x}) = \left(\frac{p}{nd}\sum_{i=1}^{n}x_i^p\right)^{\frac{1}{p}}.
\end{align}
It is easy to verify that indeed $l''(\hat{\theta}_n(\boldsymbol{x});\boldsymbol{x}) = -n\frac{pd}{[\hat{\theta}_n(\boldsymbol{x})]^2} < 0$, which shows that the MLE exists and is unique. The regularity conditions (R1)-(R4) are also satisfied and using Corollary \ref{Theoremdeltaexp} for $\epsilon = \frac{\theta_0}{2}$ we obtain
\begin{align}
\label{boundgen.gammadelta}
\nonumber & \left|{\rm E}\left[h\left(\sqrt{n\,i(\theta_0)}\left(\hat{\theta}_n(\boldsymbol{X})- \theta_0\right)\right)\right] - {\rm E}\left[h(Z)\right]\right| \leq \frac{\|h'\|}{\sqrt{n}}\left(2 + \left(3+6\frac{p}{d}\right)^{\frac{3}{4}}\right)\\
\nonumber & + \left(1 - 2\left(\frac{p}{nd}\right)^{\frac{1}{p}}\frac{\Gamma\left(\frac{nd+1}{p}\right)}{\Gamma\left(\frac{nd}{p}\right)}+\left(\frac{p}{nd}\right)^{\frac{2}{p}}\frac{\Gamma\left(\frac{nd+2}{p}\right)}{\Gamma\left(\frac{nd}{p}\right)}\right)\mathbbm{1}\left\lbrace \left\lbrace d\neq 1\right\rbrace \cup \left\lbrace p\neq 1\right\rbrace\right\rbrace\\
& \;\;\times\left[8\|h\| + \frac{\|h'\|\sqrt{ndp}|p-1|}{2}\left(\frac{1}{2^{p-2}}\mathbbm{1}\{p<2\} + \left(\frac{3}{2}\right)^{p-2}\mathbbm{1}\{p\geq2\}\right)\right].
\end{align}
Let us briefly show how to obtain this bound. For the Generalised Gamma distribution, $D(\theta_0) = q(\theta_0) = \frac{d}{p}\theta_0^p$ and thus   ${\rm E}\left[\left|T(X)-D(\theta_0)\right|^3\right] = {\rm E}\left[\left|X^p - \frac{d}{p}\theta_0^p\right|^3\right]$. This third absolute moment is very complicated to calculate. Therefore, we use H\"{o}lder's inequality and the fact  $X \sim {\rm GG}(\theta_0,d,p)\Rightarrow X^p \sim {\rm Gamma}\left(\frac{d}{p}, \frac{1}{\theta_0^p}\right)$ to get 
\begin{align}
\label{holder_for_gen.gamma}
\nonumber &{\rm E}\left[\left|X^p - \frac{d}{p}\theta_0^p\right|^3\right] \leq \left[{\rm E}\left[\left(X^p - \frac{d}{p}\theta_0^p\right)^4\right]\right]^{\frac{3}{4}}\\
\nonumber &= \left[{\rm E}\left[X^{4p}\right] + \left(\frac{d\theta_0^p}{p}\right)^4 + 6\left(\frac{d\theta_0^p}{p}\right)^2{\rm E}\left[X^{2p}\right] - 4\left(\frac{d\theta_0^p}{p}\right)^3{\rm E}\left[X^p\right] - 4\frac{d}{p}\theta_0^p{\rm E}\left[X^{3p}\right]\right]^{\frac{3}{4}}\\
& = \left[\theta_0^{4p}\frac{d}{p}\left(6+3\frac{d}{p}\right)\right]^{\frac{3}{4}} = \theta_0^{3p}\left(\frac{d}{p}\right)^{\frac{3}{4}}\left(6+3\frac{d}{p}\right)^{\frac{3}{4}}.
\end{align}
Simpler calculations yield
\begin{equation}
\label{A/k_gen_gamma}
\frac{|k'(\theta_0)|}{\sqrt{|A''(\theta_0) - k''(\theta_0)D(\theta_0)|}} = \frac{\sqrt{p}}{\sqrt{d}\theta_0^p}.
\end{equation}
Using that $X_i^p \sim {\rm Gamma}\left(\frac{d}{p},\frac{1}{\theta_0^p}\right) \Rightarrow \sum_{i=1}^{n}X_i^p \sim {\rm Gamma}\left(n\frac{d}{p},\frac{1}{\theta_0^p}\right)$, we get
\begin{align}
\label{MSEgen_gamma}
\nonumber &{\rm E}\left[\left(\hat{\theta}_n(\boldsymbol{X}) - \theta_0\right)^2\right] = \left(\frac{p}{nd}\right)^{\frac{2}{p}}{\rm E}\left[\left(\sum_{i=1}^{n}X_i^p\right)^{\frac{2}{p}}\right] + \theta_0^2 - 2\left(\frac{p}{nd}\right)^{\frac{1}{p}}\theta_0{\rm E}\left[\left(\sum_{i=1}^{n}X_i^p\right)^{\frac{1}{p}}\right]\\
\nonumber& = \theta_0^2\left(\frac{p}{nd}\right)^{\frac{2}{p}}\frac{\Gamma\left(\frac{nd+2}{p}\right)}{\Gamma\left(\frac{nd}{p}\right)} + \theta_0^2 - 2\theta_0^2\left(\frac{p}{nd}\right)^{\frac{1}{p}}\frac{\Gamma\left(\frac{nd+1}{p}\right)}{\Gamma\left(\frac{nd}{p}\right)}\\
 & = \theta_0^2\left(1 - 2\left(\frac{p}{nd}\right)^{\frac{1}{p}}\frac{\Gamma\left(\frac{nd+1}{p}\right)}{\Gamma\left(\frac{nd}{p}\right)} + \left(\frac{p}{nd}\right)^{\frac{2}{p}}\frac{\Gamma\left(\frac{nd+2}{p}\right)}{\Gamma\left(\frac{nd}{p}\right)}\right).
\end{align}
Regarding $\underset{\theta:|\theta - \theta_0|\leq \epsilon}{\sup}|D''(\theta)|$, one has to be careful  as the supremum depends on the value of $p$:
\begin{align}
\label{sup_D_gen.Gamma}
\nonumber \sup_{\theta:|\theta-\theta_0|\leq\epsilon}\left|d(p-1)\theta^{p-2}\right| &= d|p-1|\sup_{\theta:|\theta-\theta_0|\leq \epsilon}\left|\theta^{p-2}\right|\\
&= d|p-1|\begin{cases}
    (\theta_0-\epsilon)^{p-2}, & \text{if $0 < p < 2$}\\\\
    (\theta_0 + \epsilon)^{p-2}, & \text{if $p \geq 2$}.
  \end{cases}
\end{align}
Thus, applying now the results of \eqref{holder_for_gen.gamma}, \eqref{A/k_gen_gamma}, \eqref{MSEgen_gamma} and \eqref{sup_D_gen.Gamma} on the general expression for the upper bound in Corollary \ref{Theoremdeltaexp}, we obtain the result in \eqref{boundgen.gammadelta}.
\begin{remark}
\textbf{(1)} The bound in \eqref{boundgen.gammadelta} is $\mathcal{O}\left(\frac{1}{\sqrt{n}}\right)$. This is not obvious as we need to comment on the order of the term $\left(1 - 2\left(\frac{p}{nd}\right)^{\frac{1}{p}}\frac{\Gamma\left(\frac{nd+1}{p}\right)}{\Gamma\left(\frac{nd}{p}\right)}+\left(\frac{p}{nd}\right)^{\frac{2}{p}}\frac{\Gamma\left(\frac{nd+2}{p}\right)}{\Gamma\left(\frac{nd}{p}\right)}\right)$. Using the following Taylor expansion for a ratio of Gamma functions (see \cite{TE51})
$$
\frac{\Gamma(z+a)}{\Gamma(z+b)}=z^{a-b}\left(1+\frac{(a-b)(a+b-1)}{2z}+O\left(|z|^{-2}\right)\right)
$$
for large $z$ (here, $nd/p$) and bounded $a$ and $b$, we can see that this term is of order $\frac{1}{n}$, leading to the overall order of $\frac{1}{\sqrt{n}}$.\\
\textbf{(2)} The indicator function in \eqref{boundgen.gammadelta} comes from the fact that $q(\theta) = \theta \,\forall\theta\in\Theta\Longleftrightarrow d,p=1$.
\end{remark}
\subsection{Bounds for the (negative) exponential distribution}
\label{sec:exoneg}
In this subsection, we consider the most famous special case of the Generalised Gamma distribution: the (negative) exponential distribution.   First we will treat the canonical form of the distribution and then we will change the parameterisation to discuss the more interesting non-canonical setting.
\subsubsection{The canonical case: Exp$(\theta)$}
\label{subsec:exponentialcan}
We start with $X_1, \ldots, X_n$ exponentially distributed i.i.d. random variables with scale parameter $\theta>0$ and probability density function $f(x|\theta) = \theta{\rm exp}\{-\theta x\} = {\rm exp}\{\log {\theta} - \theta x\}$ for $x>0$, which we write $Exp(\theta)$. In terms of~\eqref{density_exp_family}, this means  $B = (0,\infty)$, $ \Theta = (0,\infty)$, $T(x) = -x$, $k(\theta)=\theta$, $A(\theta) = -\log{\theta}$ and $S(x)=0$. Further we have that
\begin{equation}
\nonumber l'(\theta;\boldsymbol{x}) = \frac{n}{\theta} - \sum_{i=1}^{n}x_i, \quad l''(\theta;\boldsymbol{x}) = -\frac{n}{\theta^2},
\end{equation}
the unique MLE is given by $\hat{\theta}_n(\boldsymbol{X}) = 1/\bar{X}$ with $\bar{X}=\frac{1}{n}\sum_{i=1}^nX_i$ and (R1)-(R4) are satisfied. 

With this in hand, we can easily see that $D(\theta_0) = q(\theta_0) = \frac{A'(\theta_0)}{k'(\theta_0)} =- \frac{1}{\theta_0}$ and \begin{equation}
\label{A,k_exp}
\frac{|k'(\theta_0)|}{\sqrt{|A''(\theta_0) - k''(\theta_0)D(\theta_0)|}} = \theta_0.
\end{equation} 
Simple  calculations allow us here to bypass the H\"older inequality used for the Generalized Gamma case and to obtain  ${\rm E}\left[|T(X)-D(\theta_0)|^3\right] = {\rm E}\left[\left|\frac{1}{\theta_0} - X\right|^3\right] \leq \frac{2.41456}{\theta_0^3}$. Since $X_i \sim {\rm Exp}(\theta), \; \forall i \in \left\lbrace 1,2,\ldots, n \right\rbrace$ then $\bar{X} \sim {\rm Gam}(n,n\theta)$, with ${\rm Gam}(\alpha, \beta)$ being the Gamma distribution with shape parameter $\alpha$ and rate parameter $\beta$. Consequently,
\begin{equation}
\nonumber {\rm E}[(\hat{\theta}_n(\boldsymbol{X}) - \theta_0)^2] = \frac{(n\theta_0)^2}{(n-1)(n-2)} - \frac{2n\theta_{0}^{2}}{n-1} + \theta_{0}^{2} = \frac{(n+2)\theta_{0}^{2}}{(n-1)(n-2)}.
\end{equation}
Moreover, for $\epsilon > 0$ such that $0 < \epsilon < \theta_0$, we obtain  $\underset{\theta:|\theta-\theta_0|\leq\epsilon}{\sup}\left|D''(\theta)\right| = \frac{2}{(\theta_0-\epsilon)^3}$. Choosing $\epsilon = \frac{\theta_0}{2}$, we get  $\underset{\theta:|\theta-\theta_0|\leq\epsilon}{\sup}\left|D''(\theta)\right| = \frac{16}{\theta_0^3}$. Using this result and \eqref{A,k_exp}, Corollary \ref{Theoremdeltaexp} gives 
\begin{align}
\label{boundexponentialdelta}
\nonumber \left|{\rm E}\left[h\left(\sqrt{n\,i(\theta_0)}\left(\hat{\theta}_n(\boldsymbol{X})- \theta_0\right)\right)\right] - {\rm E}\left[h(Z)\right]\right| \leq & 4.41456\frac{\|h'\|}{\sqrt{n}} + 8\|h\|\frac{(n+2)}{(n-1)(n-2)}\\
&+ 8\|h'\|\frac{\sqrt{n}(n+2)}{(n-1)(n-2)}.
\end{align}
This bound is of order $\mathcal{O}\left(\frac{1}{\sqrt{n}}\right)$ and coincides, as discussed in Remark~\ref{remark_canonical_exponential}, with the AR-bound.\\

\subsubsection{The non-canonical case: Exp$\left(\frac{1}{\theta}\right)$}
\label{sec:examplenoncanonical}
We now proceed to examine the more interesting case where $X_1, \ldots, X_n$ are i.i.d. random variables from $Exp\left(\frac{1}{\theta}\right)$. The probability density function is
\begin{align*}
f(x|\theta) &= \frac{1}{\theta}{\rm exp}\left\lbrace-\frac{1}{\theta}x\right\rbrace = {\rm exp}\left\lbrace-{\rm log}\theta - \frac{1}{\theta}x\right\rbrace
\end{align*}
corresponding to $B = (0,\infty)$, $ \Theta = (0,\infty)$, $T(x) = x$, $k(\theta) = -\frac{1}{\theta}$, $A(\theta) = {\rm log}\theta$ and $S(x) = 0$. As before, simple steps give that the MLE exists, it is unique and equal to $\hat{\theta}_n(\boldsymbol{X}) = \bar{X}$. The regularity conditions are satisfied and for $\epsilon = \frac{\theta_0}{2}$ we obtain using Corollary \ref{Theoremdeltaexp} that
\begin{align}
\label{boundnoncanexponentialdelta} \left|{\rm E}\left[h\left(\sqrt{n\,i(\theta_0)}\left(\hat{\theta}_n(\boldsymbol{X})- \theta_0\right)\right)\right] - {\rm E}\left[h(Z)\right]\right| \leq & 4.41456\frac{\|h'\|}{\sqrt{n}}.
\end{align}
Indeed, $D(\theta_0) = q(\theta_0)=\theta_0$, making the last two terms of the bound in Corollary \ref{Theoremdeltaexp} vanish. The result then follows from  ${\rm E}\left[\left|T(X) - D(\theta_0)\right|^3\right] = {\rm E}\left[\left|X - \theta_0\right|^3\right] \leq 2.41456\theta_0^3$ and
\begin{equation}
\nonumber \frac{|k'(\theta_0)|}{\sqrt{|A''(\theta_0) - k''(\theta_0)D(\theta_0)|}} = \frac{1}{\theta_0}. 
\end{equation}
\begin{remark}
\textbf{(1)} The order of the bound in terms of the sample size is, as expected, $\frac{1}{\sqrt{n}}$, corresponding to the order obtained for the Generalized Gamma distribution. The constant here is better than the one inherited from~\eqref{boundgen.gammadelta} for $p=d=1$, thanks to a sharper bound for ${\rm E}\left[\left|T(X) - D(\theta_0)\right|^3\right]$.\\
\textbf{(2)} The AR-bound is given by
$$
4.41456\frac{\|h'\|}{\sqrt{n}}+8\frac{\|h\|}{n}+2\frac{\|h'\|}{\sqrt{n}}+80\frac{\|h'\|}{\sqrt{n}}\left(\frac{6}{n}+3\right)^{1/2},
$$
showing that our new bound is an improvement.
\end{remark}
\subsubsection{Empirical results}
\label{subsec:empiricalcanonical}
For a more complete picture, we also assess the accuracy of our results using simulation-based data. The process we follow is quite simple. We generate 10000 trials of $n=10, 100, 1000, 10000$ and $100000$ random \mbox{i.i.d.} observations from the exponential distribution  Exp$\left(\frac{1}{2}\right)$ (non-canonical case).  As function $h$ we choose $h(x) = \frac{1}{x^2+2}$ with $h \in H$, $\|h\| = 0.5$ and $\|h'\| = \frac{3\sqrt{1.5}}{16}$. Simple calculations yield ${\rm E}[h(Z)] = 0.379$ and each trial gives an MLE $\hat\theta_n(\Xb)$, hence we have 10000 empirical values of $h\left(\sqrt{n\,i(\theta_0)}\left(\hat{\theta}_n(\boldsymbol{X}) - \theta_0\right)\right)$ to compare to 0.379. Taking the average hence provides a simulated estimation of $\left|{\rm E}\left[h\left(\sqrt{n\,i(\theta_0)}\left(\hat{\theta}_n(\boldsymbol{X}) - \theta_0\right)\right)\right] - {\rm E}[h(Z)]\right|$, which we compare to the upper bound given in  \eqref{boundnoncanexponentialdelta}.  Our bound provides a very strong improvement on the AR-bound (see Table \ref{tableresultexponentialnoncanonicaldelta}). Of course, this estimated distance is only a lower bound to the true distance, as we have chosen a particular function $h$ instead of the supremum over all functions $h\in H$, but its calculation still provides an idea of the accuracy of our bounds. This closeness logically increases with the sample size and becomes quite sharp for $n\geq100$. 

\begin{table}[h]
\caption{Simulation results for the Exp$\left(\frac{1}{2}\right)$ distribution treated as a non-canonical exponential family}
\vspace{0.04in}
\centering
\begin{tabular}{r|r|r|r}
	  $n$ & $\left|{\rm \hat{E}}\left[ h\left(\sqrt{n\:i(\theta_0)}(\hat{\theta}_n(\boldsymbol{X}) - \theta_0) \right)\right]  - {\rm E}[h(Z)]\right|$ & New Bound & AR-bound\\
	  \hline
	  \hline
  10 & 0.0034 & 0.321  & 11.888\\
  \hline
  100 & 0.0022 & 0.101  & 3.401\\
  \hline
  1000 & 0.0012 & 0.032 & 1.058\\
  \hline
  10000 & 0.0008 & 0.010  & 0.333\\
  \hline
  100000 & 0.0004 & 0.003 & 0.105
  \end{tabular}
\label{tableresultexponentialnoncanonicaldelta}
  \end{table}

\

\noindent ACKNOWLEDGMENTS\vspace{2mm}

\noindent The authors would like to thank Gesine Reinert for various insightful comments and suggestions.

\bibliographystyle{alea3}
\bibliography{AL16}

\end{document}